\documentclass[10pt]{article}

\usepackage{latexsym}
\usepackage{amsmath,amsthm}
\usepackage{amssymb}
\usepackage{bm}
\usepackage{graphicx}
\usepackage{wrapfig}
\usepackage{fancybox}
\usepackage{hyperref}
\usepackage{color}
\usepackage{chngcntr}

\makeatletter
\def\@seccntformat#1{\@ifundefined{#1@cntformat}%
   {\csname the#1\endcsname\quad}  
   {\csname #1@cntformat\endcsname}
}
\let\oldappendix\appendix 
\renewcommand\appendix{%
    \oldappendix
    \newcommand{\section@cntformat}{\appendixname~\thesection:\quad}
}
\makeatother

\newtheorem{lemma}{Lemma}[section]
\newtheorem{theorem}{Theorem}[section]
\newtheorem{corollary}{Corollary}[section]

\usepackage[margin=1in, centering]{geometry}
\usepackage{fancyhdr}
\usepackage[absolute]{textpos}
\TPMargin{5pt}

\begin{document}

\title{Some results on the $\mathbf{\xi(s)}$ and $\mathbf{\Xi(t)}$ functions associated with Riemann's $\mathbf{\zeta(s)}$ function.
\footnote{This paper was originally posted on \url{http://hp.hisashikobayashi.com/\#5}  January 22, 2016, under the title ``No. 5. Some results on the $\mathbf{\xi(s)}$ and $\mathbf{\Xi(t)}$ functions associated with Riemann's $\mathbf{\zeta(s)}$ function: Towards a Proof of the Riemann Hypothesis.''}}
\date{2016/03/04}
\author{Hisashi Kobayashi\footnote{Professor Emeritus, Department of Electrical Engineering, Princeton University, Princeton, NJ 08544.}}

\maketitle

\begin{abstract}
We report on some properties of the $\xi(s)$ function and its value on the critical line, $\Xi(t)=\xi\left(\tfrac{1}{2}+it\right)$. First, we present some identities that hold for the log derivatives of a holomorphic function. We then re-examine Hadamard's product-form representation of the $\xi(s)$ function, and present a simple proof of the horizontal monotonicity of the modulus of $\xi(s)$.  We then show that the $\Xi(t)$ function can be interpreted as the autocorrelation function of a weakly stationary random process, whose power spectral function $S(\omega)$ and $\Xi(t)$ form a Fourier transform pair.  We then show that $\xi(s)$ can be formally written as the Fourier transform of $S(\omega)$ into the complex domain $\tau=t-i\lambda$, where $s=\sigma+it=\tfrac{1}{2}+\lambda+it$. We then show that the function $S_1(\omega)$ studied by P\'{o}lya has $g(s)$ as its Fourier transform, where $\xi(s)=g(s)\zeta(s)$. Finally we discuss the properties of the function $g(s)$, including its relationships to Riemann-Siegel's $\vartheta(t)$ function, Hardy's Z-function, Gram's law and the Riemann-Siegel asymptotic formula.\\

\end{abstract}
\section{Definition of $\mathbf{\xi(s)}$ and its properties}

The Riemann zeta function is defined by
\begin{align}
\zeta(s)=\sum_{n=1}^\infty n^{-s}, ~~\mbox{for}~~\Re(s)>1,
\end{align}
which is then defined for the entire $s$-domain by analytic continuation (See Riemann \cite{riemann} and Edwards \cite{edwards}). In this article we investigate some properties of the function $\xi(s)$
\footnote{In Riemann's 1859 seminal paper \cite{riemann} he was primarily concerned with the properties of this function evaluated on the critical line $s=\tfrac{1}{2}+it$, which he denoted as $\xi(t)$.  We write it as $\Xi(t)$ instead, as defined in (\ref{def-Xi}). See, e.g., Titchmarsh \cite{titchmarsh} p. 16.  Edwards \cite{edwards} writes explicitly $\xi(\tfrac{1}{2}+it)$ for $\Xi(t)$.}
defined by (see Appendix \ref{append-A})
\begin{align}\label{xi-def}
\xi(s)=\frac{s(s-1)}{2}\pi^{-s/2}\Gamma(s/2)\zeta(s).
\end{align}
The function $\xi(s)$ is an \emph{entire} function with the following ``reflective'' property:
\begin{align}\label{symmetric_property}
\xi(1-s)=\xi(s).
\end{align}
If we write
\[s=\sigma+it=\tfrac{1}{2}+\lambda+it,\]
the property (\ref{symmetric_property}) is paraphrased as
\begin{align}
      \Re\left\{\xi\left(\tfrac{1}{2}+\lambda+it\right)\right\}&=\Re\left\{\xi\left(\tfrac{1}{2}-\lambda+it\right)\right\}, \label{A-symmetric}\\
      \Im\left\{\xi\left(\tfrac{1}{2}+\lambda+it\right)\right\}&=-\Im\left\{\xi\left(\tfrac{1}{2}-\lambda+it\right)\right\}, \label{B-skew}
\end{align}
By setting $\lambda=0$ in (\ref{B-skew}), we find
\begin{align}\label{B=0}
\Im\left\{\xi\left(\tfrac{1}{2}+it\right)\right\}=0,  ~~\mbox{for all}~~t,
\end{align}
which implies that $\xi(s)$ is real on the ``critical line.''  Thus, if we define a real-valued function
\begin{align}\label{def-Xi}
\Xi(t)=\xi\left(\tfrac{1}{2}+it\right)=\Re\left\{\xi\left(\tfrac{1}{2}+it\right)\right\},
\end{align}
the Riemann hypothesis can be paraphrased as ``The zeros of $\Xi(t)$ are all real,'' which is indeed the way Riemann stated his conjecture, now known as the \emph{Riemann hypothesis} or RH for short.

By applying Laplace's equation to $\Im\left\{\xi(s)\right\}$ and using (\ref{B=0}), we readily find
\begin{align}
\left.\frac{\partial^2 \Im\left\{\xi(s)\right\}}{\partial\lambda^2}\right|_{\lambda=0}=0.
\end{align}
Thus, it follows that $\Im\left\{\xi(s)\right\}$ must be a polynomial in $\lambda$ of degree 1 in the
vicinity of $\lambda=0$, viz.,
\begin{align}\label{B-straight-line}
 \Im\left\{\xi(s)\right\}\approx b(t)\lambda,~~\mbox{for}~\lambda\approx 0,
\end{align}
where $b(t)$ is a function of $t$ only, independent of $\lambda$.

Similarly, by applying Laplace's equation to $\Re\left\{\xi(s)\right\}$ and using the Cauchy-Riemann equation:
\begin{align}
\frac{\partial\Re\left\{\xi(s)\right\}}{\partial t}=-\frac{\partial\Im\left\{\xi(s)\right\}}{\partial\lambda}.
\end{align}
and using (\ref{B-straight-line}), we find that the real part of $\xi(s)$ is a polynomial in $\lambda$ of degree 2:
\begin{align}\label{A-parabola}
\Re\left\{\xi(s)\right\}\approx\frac{b'(t)}{2}\lambda^2,~~\mbox{for}~\lambda\approx 0,
\end{align}
where $\displaystyle{b'(t)=\frac{db(t)}{dt}}$.

\section{Preliminaries} \label{sec-preliminaries}

\subsection{Logarithmic Differentials of Holomorphic Functions}

We begin with the following lemma that is applicable to any holomorphic function.
\begin{lemma}\label{lemma-1}
For a holomorphic function $f(s)$ we have
\begin{align}
\frac{1}{|f(s)|}\cdot \frac{\partial |f(s)|}{\partial \sigma} &=\Re\left\{\frac{f'(s)}{f(s)}\right\},\label{Re-f'_f}\\
\frac{1}{|f(s)|}\cdot \frac{\partial |f(s)|}{\partial t}&=-\Im\left\{\frac{f'(s)}{f(s)}\right\},\label{Im-f'_f}
\end{align}
wherever $f(s)\neq 0$, where $\displaystyle{f'(s)=\frac{d f(s)}{ds}}$.
\end{lemma}
\begin{proof}
See Kobayashi \cite{kobayashi-report3}.
\end{proof}
By differentiating the logarithm of $f(s)$ further, we obtain
\begin{corollary}\label{corollary-2nd-dif}
For the holomorphic function $f(s)$ of Lemma \ref{lemma-1} the following identities also hold:
\begin{align}
\frac{1}{|f(s)|}\frac{\partial^2 |f(s)|}{\partial\sigma^2}-\left(\frac{1}{|f(s)|} \frac{\partial |f(s)|}{\partial\sigma}\right)^2 &=\Re\left\{\frac{f''(s)}{f(s)}-\left(\frac{f'(s)}{f(s)}\right)^2\right\},\label{dif-log-f-wrt-sigma-twice}\\
\frac{1}{|f(s)|}\frac{\partial^2 |f(s)|}{\partial t^2}-\left(\frac{1}{|f(s)|}\frac{\partial |f(s)|}{\partial t}\right)^2 &=-\Re\left\{\frac{f''(s)}{f(s)}-\left(\frac{f'(s)}{f(s)}\right)^2\right\}.\label{dif-log-f-wrt-t-twice}
\end{align}
wherever $f(s)\neq 0$, where $\displaystyle{f''(s)=\frac{d^2 f(s)}{ds^2}}$.
\end{corollary}
\begin{proof}
See Kobayashi \cite{kobayashi-report3}.
\end{proof}

\subsection{The Product Formula for $\mathbf{\xi(s)}$}\label{sec-Hadamard}

Hadamard \cite{hadamard} obtained in 1893 the following product-form representation
\begin{align}\label{Hadamard-product}
\xi(s)=\tfrac{1}{2}e^{Bs}\prod_n \left[\left(1-\frac{s}{\rho_n}\right)e^{s/\rho_n}\right],
\end{align}
using Weirstrass's factorization theorem, which asserts that any entire function can be represented by a product involving its zeroes.
In (\ref{Hadamard-product}), the product is taken over all (infinitely many) zeros $\rho_n$'s of the function $\xi(s)$, and $B$ is a real
constant.  Detailed accounts of this formula are found in many books (see e.g., Edwards \cite{edwards}, Iwaniec \cite{iwaniec}
Patterson \cite{patterson} and Titchmarsh \cite{titchmarsh}).
Sondow and Dumitrescu \cite{sondow} and Matiyasevich et al. \cite{matiyasevich} explored the use of the above product
form, hoping to find a possible proof of the Riemann hypothesis.

By taking the logarithm of (\ref{Hadamard-product}) and differentiating it, we obtain
\begin{align}\label{Hadamard-log-dif}
\frac{\xi'(s)}{\xi(s)}=B+\sum_n\left(\frac{1}{s-\rho_n}+\frac{1}{\rho_n}\right).
\end{align}
From the definition of $\xi(s)$ in (\ref{xi-def}), we have
\begin{align}\label{log-dif-xi}
\frac{\xi'(s)}{\xi(s)}=\frac{1}{s}+\frac{1}{s-1}-\frac{\log\pi}{2}+\Psi(\tfrac{s}{2})+\frac{\zeta'(s)}{\zeta(s)},
\end{align}
where
\begin{align*}
\Psi(s)=\frac{\Gamma'(s)}{\Gamma(s)}
\end{align*}
is the digamma function.

We equate (\ref{Hadamard-log-dif}) to (\ref{log-dif-xi}), use the identity $\displaystyle{\Psi(\tfrac{s}{2}+1)=\tfrac{1}{s}+\tfrac{1}{2}\Psi(\tfrac{s}{2})}$, and set $s=0$, obtaining
\begin{align}
B+\sum_n\left(-\frac{1}{\rho_n}+\frac{1}{\rho_n}\right)=-1-\tfrac{1}{2}+\tfrac{1}{2}\Psi(1)+\frac{\zeta'(0)}{\zeta(0)}.
\end{align}
By using $\zeta'(0)/\zeta(0)=\log(2\pi)$, and $\Psi(1)=\Gamma'(1)=-\gamma$ (where
$\gamma=0.5772218\ldots$ is the Euler constant), we determine the constant $B$  as
\begin{align}\label{constant-B}
B=\log(2\pi)-1-\tfrac{1}{2}\log\pi-\gamma/2=\tfrac{1}{2}\log(4\pi)-1-\gamma/2=-0.0230957\ldots.
\end{align}

Davenport (\cite{davenport} pp. 81-82) derives an alternative expression for $B$. The reflective property of
$\xi(s)$ gives the identity
\begin{align}
\frac{\xi'(s)}{\xi(s)}=-\frac{\xi'(1-s)}{\xi(1-s)},
\end{align}
which, together with (\ref{Hadamard-log-dif}), yields
\begin{align}\label{Eqn-for-B}
B+\sum_n\left(\frac{1}{s-\rho_n}+\frac{1}{\rho_n}\right)=-B-\sum_n\left(\frac{1}{1-s-\rho_n}+\frac{1}{\rho_n}\right).
\end{align}
Thus,
\begin{align}\label{B-def}
B&=-\sum_n\frac{1}{\rho_n}-\frac{1}{2}\left(\sum_n\frac{1}{s-\rho_n}-\sum_n\frac{1}{s-(1-\rho_n)}\right)\nonumber\\
&=-\sum_n\frac{1}{\rho_n}=-2\sum_{n=1}^\infty\frac{\sigma_n}{\sigma_n^2+t_n^2},
\end{align}
Note that the two summed terms in the parenthesis in the first line of the above cancel to each other, because if $\rho_n$ is a zero, so is $1-\rho_n$.
To obtain the final expression in the above, we use the property that when $\rho_n=\sigma_n+it_n$ is a zero, so is its complex conjugate $\rho_n^*=\sigma_n-it_n$, thus we enumerate zeros in such a way that
 $\rho_n^*=\rho_{-n}$.

By substituting (\ref{B-def}) back into (\ref{Hadamard-product}), we obtain
\begin{align}\label{simple-product}
\xi(s)&=\tfrac{1}{2}\exp\left(-s\sum_n \frac{1}{\rho_n}\right) \prod_n\left(1-\frac{s}{\rho_n}\right)e^{s/\rho_n}
=\tfrac{1}{2}\prod_n e^{-s/\rho_n}\left(1-\frac{s}{\rho_n}\right)e^{s/\rho_n}\nonumber\\
&=\tfrac{1}{2}\prod_n\left(1-\frac{s}{\rho_n}\right).
\end{align}
This is nothing but the product form
\begin{align*}
 \xi(s)=\xi(0)\prod_n\left(1-\frac{s}{\rho_n}\right),
\end{align*}
which Edwards (see \cite{edwards} p. 18 and pp. 46-47) attributes to Riemann.

Then, Eqn.(\ref{Hadamard-log-dif}) is simplified to
\begin{align}\label{dif-log-simple-product}
\frac{\xi'(s)}{\xi(s)}=\sum_n\frac{1}{s-\rho_n}.
\end{align}
From this and Lemma \ref{lemma-1}, we have
\begin{align}\label{sigma-sigma_n}
\frac{1}{|\xi(s)|}\frac{\partial|\xi(s)|}{\partial \sigma}
&=\Re\left(\sum_n \frac{1}{s-\rho_n}\right)=\sum_n\frac{\sigma-\sigma_n}{(\sigma-\sigma_n)^2+(t-t_n)^2}.
\end{align}
Thus, we arrive at the following theorem concerning the monotonicity of the $|\xi(s)|$ function, which Sondow and
Dumitrescu \cite{sondow} proved in a little more complicated way based on (\ref{Hadamard-product}) instead of
(\ref{simple-product}).  Matiyaesevich et al. \cite{matiyasevich} also discuss the monotonicity of the $\xi(s)$ and
other functions.
%
%
\begin{theorem}[Monotonicity of Modulus Function $|\xi(s)|$]\label{theorem-modulus-monotonicity}
Let $\sigma_{\sup}$ be the supremum of the real parts of all zeros:
\[ \sigma_{\sup}=\sup_n\{\sigma_n\}.\]
Then the modulus $|\xi(\sigma+it)|$ is a monotone increasing function of $\sigma$ in the region $\sigma> \sigma_{\sup}$
for all real $t$. Likewise, the modulus is a monotone decreasing function of $\sigma$ in the region
$\sigma<\sigma_{\inf}$,
where
\[ \sigma_{\inf}=\inf_n\{ \sigma_n\}=1-\sigma_{\sup}.\]
\begin{proof}
It is apparent from (\ref{sigma-sigma_n}) that $|\xi(s)|$ is a monotone increasing function of $\sigma$ in the range
$\sigma>\sigma_{\sup}\geq \tfrac{1}{2}$ for all $t$.  Because of the reflective property (\ref{symmetric_property}) it then readily follows that $|\xi(s)|$ is a monotone decreasing function of $\sigma$ in the range $\sigma<1-\sigma_{\sup}\leq \tfrac{1}{2}$.
\end{proof}

\end{theorem}

Thus, if all zeta zeros are located on the critical line, i.e., if $\sigma_{\sup}=\sigma_{\inf}=\tfrac{1}{2}$, the derivative of the modulus $|\xi(s)|$ is positive for $\sigma>\tfrac{1}{2}$, and negative for $\sigma<\tfrac{1}{2}$. Thus, we have shown the necessity of
monotonicity of the modulus function $|\xi(s)|$, which has been one of major concerns towards a proof of the Riemann
hypothesis.
\begin{corollary}[Monotonicity of Modulus Function $|\xi(s)|$, if the Riemann hypothesis is true]\label{corollary-modulus-monotonicity}
If all zeta zeros are on the critical line, the modulus $|\xi(\sigma+it)|$ is a monotone increasing function of
$\sigma$ in the right half plane, $\sigma>\tfrac{1}{2}$.
Likewise, the modulus is a monotone decreasing function of $\sigma$ in the left half plane, $\sigma<\tfrac{1}{2}$.
\begin{proof}
The above discussion that has led to this corollary should suffice as a proof.
\end{proof}
\end{corollary}

\subsection{Functions $\mathbf{a(\lambda,t), b(\lambda,t), \alpha(\lambda,t), \beta(\lambda,t)}$ and Their Properties}\label{sub-sec-a(t)}

Take the imaginary part of both sides of (\ref{dif-log-simple-product}) and set
$s=\tfrac{1}{2}+it$. By noting that $\xi(s)$ is real for $\sigma=\tfrac{1}{2}$, we obtain
\begin{align}
\left.\frac{1}{\xi(s)}\frac{\partial \Im(\xi(s))}{\partial \sigma}\right|_{\sigma=\tfrac{1}{2}}
=\sum_n\frac{t-t_n}{(t-t_n)^2+(\tfrac{1}{2}-\sigma_n)^2}.
\end{align}
Recall the function $b(t)$ defined in (\ref{B-straight-line}). Then, the LHS of the above is
$\displaystyle{\frac{b(t)}{\Xi(t)}}$, where
\begin{align}
\Xi(t)&=\xi\left(\tfrac{1}{2}+it\right)=\frac{1}{2}\prod_n\left(1-\frac{\tfrac{1}{2}+it}{\sigma_n+it_n}\right)
                                                  \label{expression-c(t)}\\
b(t)&=\left.\frac{\partial \Im\left\{\xi(s)\right\}}{\partial\sigma}\right|_{\sigma=\tfrac{1}{2}} =\Xi(t)\cdot\sum_n\frac{t-t_n}{(t-t_n)^2+(\tfrac{1}{2}-\sigma_n)^2}.\label{expression-b(t)}
\end{align}
Differentiate (\ref{dif-log-simple-product}) once more, and we obtain
\begin{align*}
\frac{\xi''(s)\xi(s)-{\xi'}^2(s)}{\xi^2(s)}=-\sum_n\frac{1}{(s-\rho_n)^2},
\end{align*}
which can be rearranged to yield
\begin{align}\label{2nd-derivative}
\frac{\xi''(s)}{\xi(s)}=\left(\frac{\xi'(s)}{\xi(s)}\right)^2-\sum_n\frac{1}{(s-\rho_n)^2}.
\end{align}
Taking the real part of both sides, and evaluating them at $s=\tfrac{1}{2}+it$, we find
\begin{align}\label{2a/c}
\frac{2a(t)}{\Xi(t)}
&=-\left(\frac{b(t)}{\Xi(t)}\right)^2+\sum_n\frac{(t-t_n)^2-(\tfrac{1}{2}-\sigma_n)^2}
{\left[(\tfrac{1}{2}-\sigma_n)^2+(t-t_n)^2\right]^2}\nonumber\\
&=\frac{-b^2(t)^2+b'(t)\Xi(t)-b(t)\Xi'(t)}{\Xi^2(t)}.
\end{align}
where
\begin{align}\label{2a-A''}
2a(t)=\left.\frac{\partial^2 \xi(s)}{\partial\sigma^2}\right|_{\sigma=\tfrac{1}{2}}.
\end{align}

From the Cauchy-Riemann equation we find
\begin{align}\label{c-CS-eq}
 \Xi'(t)=-\left.\frac{\partial\Im\left\{\xi(s)\right\}}{\partial\sigma}\right|_{\sigma=\tfrac{1}{2}}=-b(t).
\end{align}
By substituting this into (\ref{2a/c}), we obtain a surprisingly simple result:
\begin{align}\label{a=b'-c''}
a(t)&=\tfrac{1}{2}b'(t)=-\tfrac{1}{2}\Xi''(t),
\end{align}
which can be alternatively obtained by applying the Laplace equation to (\ref{2a-A''}).

The above formulae carry over to any point $s=\tfrac{1}{2}+\lambda+it$:
\begin{lemma} \label{lemma-lambda-t}
Let us define
\begin{align}
2a(\lambda,t)&=\frac{\partial^2 \Re\left\{\xi(s)\right\}}{\partial\lambda^2}
=-\Re\left\{\xi''(t)\right\},\label{def-a(s)}\\
b(\lambda,t)&=\frac{\partial \Im\left\{\xi(s)\right\}}{\partial\lambda},\label{def-b(s)}
\end{align}
where $\xi''(s)$ is the second partial derivative of $\xi(s)$ with respect to $t$.
Then, the following relations hold:
\begin{align}
a(\lambda,t)&=\tfrac{1}{2}b'(\lambda,t), \label{a-b-relation}\\
b(\lambda,t)&=- \Re\left\{\xi'(t)\right\} \label{b-c-relation}
\end{align}
\begin{proof}
By applying the Cauchy-Riemann equations and Laplace's equation, the above relations can be easily derived.
\end{proof}
\end{lemma}

We now derive similar functions and their relations by interchanging
$\Re\left\{\xi(s)\right\}$ and
$\Im\left\{\xi(s)\right\}$.
\begin{corollary}\label{corollary-dual}
Let us define
\begin{align}
2\alpha(\lambda,t)&=\frac{\partial^2 \Im\left\{\xi(s)\right\}}{\partial\lambda^2}
=- \Im\left\{\xi''(s)\right\},\label{def-alpha)}\\
\beta(\lambda,t)&=\frac{\partial \Re\left\{\xi(s)\right\}}{\partial\lambda}.\label{def-beta}
\end{align}
Then the following relations hold:
\begin{align}
\alpha(\lambda,t)&=-\tfrac{1}{2}\beta'(\lambda,t), \label{alpha-beta}\\
\beta(\lambda,t)&=\Im\left\{\xi'(s)\right\}. \label{beta-gamma}\\
\frac{\partial a(\lambda,t)}{\partial\lambda} &=\alpha'(\lambda,t),~~~~
\frac{\partial \alpha(\lambda,t)}{\partial\lambda}=-a'(\lambda,t)\\
\frac{\partial \beta(\lambda,t)}{\partial\lambda}&= b'(\lambda,t),~~~~
\frac{\partial b(\lambda,t)}{\partial\lambda}=-\beta'(\lambda,t).
\end{align}
\begin{proof}
By applying the Cauchy-Riemann equations and Laplace's equation, the above relations can be easily derived.
\end{proof}
\end{corollary}

\section{The Fourier transform representation of $\xi(s)$ } \label{sec-Fourier-transform}
\subsection{Integral representation of $\mathbf{\xi(s)}$}
We begin with the following integral representation of $\xi(s)$ (see Appendix \ref{append-A}) found in Edwards \cite{edwards}, p.16.
\begin{align}\label{edwards-p16}
\xi(s)&=\frac{1}{2}-\frac{s(1-s)}{2}\int_1^\infty\psi(x)\left(x^{s/2}+x^{(1-s)/2}\right)\frac{dx}{x},
\end{align}
where
\begin{align}\label{func-psi}
\psi(x)=\sum_{n=1}^\infty e^{-n^2\pi x}
\end{align}
is called the \emph{theta function}.
By applying integration by parts to (\ref{edwards-p16}) and Jacobi's identity for the theta function\footnote{Jacobi's identity for the theta function $\psi(x)$ is
\begin{align}\label{Jacobi}
2\psi(x)+1=x^{-1/2}\left(2\psi(x^{-1})+1\right).
\end{align}
} Edwards (\cite{edwards}, p. 17) gives the following expression  by generalizing Riemann's result, which holds for any complex number $s$:
\begin{align}\label{xi-cosh}
\xi(s)=4\int_1^\infty\frac{d[x^{3/2}\psi'(x)]}{dx} x^{-1/4}\cosh\left[\tfrac{1}{2}\left(s-\tfrac{1}{2}\right)\log x\right]\,dx.
\end{align}
By writing
\begin{align}\label{C-x}
\frac{d[x^{3/2}\psi'(x)]}{dx}x^{-1/4}=\pi x^{1/4}D(x)
\end{align}
with $D(x)$ defined by
\begin{align}\label{D-x}
D(x)&=\sum_{n=1}^\infty n^2(n^2\pi x-\tfrac{3}{2})e^{-n^2\pi x}>0,~~\mbox{for}~~x\geq 1,
\end{align}
we can write (\ref{xi-cosh}) as
\begin{align}\label{xi-Dx}
\xi(s)&=4\pi\int_1^\infty x^{1/4}D(x)\cos\left(\frac{\tau\log x}{2}\right)\,dx,
\end{align}
where $\tau$ is a complex number defined by
\begin{align}\label{complex-time}
\tau=t-i\lambda=-i(s-\tfrac{1}{2}),
\end{align}
and we used the identity $\cosh(iy)=\cos y$. By changing the variable from $x$ to $\omega$ by
\begin{align}\label{omega}
\omega=\frac{\log x}{2},~~x\geq 1,
\end{align}
and defining
\begin{align}\label{S-w}
 S(\omega)=8\pi e^{5\omega/2} D(e^{2\omega}),~~\omega\geq 0
\end{align}
we can write (\ref{xi-Dx}) as
\begin{align}\label{xi-S}
\xi(s)&=\int_0^\infty S(\omega)\cos(\omega \tau)\,d\omega,
\end{align}
which is a compact expression for
\begin{align}
\xi(\tfrac{1}{2}+\lambda+it)=\int_0^\infty S(\omega)\left(\cos\omega t\cosh(\omega\lambda)+i\sin\omega t\sinh(\omega\lambda)\right)\,d\omega.
\end{align}
On the critical line $s=\tfrac{1}{2}+it$ (i.e., when $\lambda=0$), the above reduces to a more familiar formula
\begin{align}\label{xi-critical-line}
\Xi(t)=\int_0^\infty S(\omega)\cos(\omega t)\,d\omega.
\end{align}

\subsection{The kernel function $\mathbf{S(\omega)}$ as a power spectral function.}

The kernel $S(\omega)$ defined by (\ref{S-w}) is positive for all $\omega\geq 0$, because $D(x)$ is positive for $x\geq 1$.
Therefore, $S(\omega)$ can qualify as a \emph{spectral density function} of a certain wide-sense stationary (a.k.a. weakly stationary) process,
and we can interpret $\Xi(t)$ as its autocorrelation function (see e.g., \cite{kobayashi} p. 349).  In this context, the Fourier transforms between the spectrum $S(\omega)$ and the function $\Xi(t)$ are what is known as the Wiener-Khinchin theorem (a.k.a. the Wiener-Khinchin-Einstein theorem).  The inverse transform to (\ref{xi-critical-line}), given below by (\ref{inverse-transform}), exists when $\Xi(t)$ is absolutely integrable.

The Fourier transform representation (\ref{xi-critical-line}) has been studied by George P\'{o}lya \cite{polya} and others (see e.g., Titchmarsh \cite{titchmarsh}, Chapter 10).
Dimitrov and Rusev \cite{dimitrov} give a comprehensive review of the past work on ``zeros of entire Fourier transforms,'' including P\'{o}lya's work.

From the above observation that $S(\omega)$ is positive for $\omega\geq 0$, we can readily establish the following proposition:
\begin{theorem}
The modulus $|\Xi(t))|$ is maximum at $t=0$, i.e.,
\begin{align}
|\Xi(t)|\leq \Xi(0)=0.4971\ldots, ~~\mbox{for all}~~t.
\end{align}
Furthermore,
\begin{align}
\int_0^\infty \Xi(t)\,dt=3\pi\left(\frac{\pi^{1/4}}{\Gamma(3/4)}-1\right)=2.8067\ldots.
\end{align}
\begin{proof}
From (\ref{xi-S}), it readily follows that
\begin{align}
|\Xi(t)|\leq \int_0^\infty |S(\omega)|\,d\omega=\int_0^\infty S(\omega)\,d\omega=\Xi(0).
\end{align}
Since $\zeta(\tfrac{1}{2})=-1.46035\ldots$\footnote{See e.g. \url{https://oeis.org/A059750}.}, and
$g(\tfrac{1}{2})=-\tfrac{1}{8}\pi^{-1/4}\Gamma(\tfrac{1}{4})=-0.3404\ldots$, we have $\Xi(0)=\xi(\tfrac{1}{2})=g(\tfrac{1}{2})\zeta(\tfrac{1}{2})
=0.4971\ldots$.

From the Wiener-Khinchin inverse formula, which holds when $\Xi(t)$ is absolutely integrable, we have
\begin{align}\label{inverse-transform}
S(\omega)=\frac{2}{\pi}\int_0^\infty \Xi(t)\cos(\omega t)\,dt.
\end{align}
By setting $\omega=0$, we readily find
\begin{align}
S(0)=\frac{2}{\pi}\int_0^\infty \Xi(t)\,dt.
\end{align}
By setting $\omega=0$ in (\ref{S-w}), we have
\begin{align}
S(0)=8\pi D(1)=8\left(\tfrac{3}{2}\psi'(1)+\psi''(1)\right).
\end{align}
The function $\psi(x)$ satisfies the aforementioned  Jacobi's identity (\ref{Jacobi}).
By differentiating the identity equation, we find
\begin{align}
2\psi'(x)&=-\tfrac{1}{2}x^{-3/2}-x^{-3/2}\psi(1/x)-2x^{-5/2}\psi'(1/x)\label{difpsi1}
\end{align}
By setting $x=1$ in (\ref{difpsi1}) we obtain
\begin{align}\label{psi'-1}
\psi'(1)&=-\tfrac{1}{8}(1+2\psi(1))\\
\end{align}
The value of $\psi(1)$ is known (see e.g., Yi \cite{yi}, Theorem 5.5 in p. 398)
\begin{align}\label{psi-1}
\psi(1)=\tfrac{1}{2}\left(\frac{\pi^{1/4}}{\Gamma(3/4)}-1\right)=\tfrac{1}{2}\left(\frac{1.3313}{1.2254}-1\right)=0.0432\ldots.
\end{align}
Hence,
\begin{align}
\psi'(1)=-\tfrac{1}{8}\frac{\pi^{1/4}}{\Gamma(3/4)}=-0.1358\ldots.
\end{align}
The numerical evaluation of $\psi''(1)$ is straightforward, since its series representation converges rapidly:
\begin{align}
\psi''(1)=\pi^2\sum_{n=1}^\infty n^4e^{-\pi n^2}\approx\pi^2\sum_{n=1}^2 n^4e^{-\pi n^2}= 0.4271\ldots.
\end{align}
Thus, we finally evaluate
\begin{align}
\int_0^\infty\Xi(t)\,dt&=\frac{\pi}{2}S(0)
=4\pi\left(\tfrac{3}{2}\psi'(1)+\psi''(1)\right)= 2.8067\ldots
\end{align}
\end{proof}
\end{theorem}

The variable $t$ of the complex variable $s=\sigma+it=\tfrac{1}{2}+\lambda+it$ is often called the \emph{height} in the zeta function related literature.  In view of the Wiener-Khinchin theorem (\ref{xi-critical-line}) and (\ref{inverse-transform}), it may be appropriate to interpret $t$ as ``time'' and the variable $\omega$ of $S(\omega)$ as the ``(angular) frequency.''  Then, we may refer to the complex number $\tau$ defined by (\ref{complex-time}) as ``complex-time.''
Use of the complex-time $\tau$ allow the compact representation (\ref{xi-S}) given earlier, viz.
\begin{align}\label{xi-cos}
\xi(s)&=\int_0^\infty S(\omega)\cos(\omega\tau)\,d\omega.
\end{align}
This interpretation of Riemann's result (\ref{xi-cosh}) will shed some new light to the Fourier transform representation of the
$\xi(s)$ function.  We will further discuss this in a later section.

\section{Further results on the Fourier transform representation}

\subsection{Decomposition of $\mathbf{S(\omega)}$}

In the Fourier transform representation (\ref{xi-S})
the kernel function $S(\omega)$ can be expressed as
\begin{align}\label{S-series}
S(\omega)=\sum_{n=1}^\infty S_n(\omega),
\end{align}
with
\begin{align}\label{S-n}
S_n(\omega)=8\pi e^{5\omega/2}D_n(e^{2\omega}),
\end{align}
where
\begin{align}\label{D-func}
D_n(x)= n^2(n^2\pi x-\tfrac{3}{2})e^{-n^2\pi x}.
\end{align}

The Fourier transform can therefore be written as a summation of infinite components, i.e.,
\begin{align}\label{xi-sum-f_n}
\xi(s)=\sum_{n=1}^\infty f_n(s),
\end{align}
with
\begin{align}\label{f_n-def}
f_n(s)&=\int_0^\infty S_n(\omega)\cos(\omega\tau)\,d\omega\nonumber\\
&=8\pi\int_0^\infty e^{5\omega/2}D_n(e^{2\omega})\cos(\omega\tau)\,d\omega.
\end{align}
The switching in the order between the summation over $n$ and the integration over $\omega$, as used in (\ref{f_n-def}) and (\ref{xi-sum-f_n}), can be justified, because the series $\displaystyle{\sum_{n=1}^N S_n(\omega)}$ uniformly converges to $S(\omega)$ as $N\to\infty$ in the entire range $\omega\geq 0$.
Note also that in the range $\omega\geq 0$, $S(\omega)$ is predominantly determined by its first components $S_1(\omega)$, leaving $S_n(\omega), n\geq 2$ negligibly smaller. However, any attempt to replace $S(\omega)$ by $S_1(\omega)$ in an effort to prove the Riemann hypothesis would fail, as argued by Titchmarsh (see \cite{titchmarsh}, Chapter 10, p. 256).

\subsection{The Fourier transform of  $\mathbf{S(\omega)}$ in $\mathbf{-\infty<\omega<\omega}$.}

Now let us consider the Fourier transform of $S(\omega)$ defined over the entire real line $-\infty<\omega<\infty$, instead of the positive line $\omega\geq 0$. Note that the kernel $S(\omega)$ of (\ref{S-w}) extended to the range $-\infty<\omega<\infty$ is symmetric, i.e.,
\begin{align}\label{symmetry-S}
S(-\omega)=S(\omega),~~-\infty<\omega<\infty,
\end{align}
which can be shown using Jacobi's identity (\ref{Jacobi}). See \cite{kobayashi-report4} for a derivation of (\ref{symmetry-S}).

The Fourier transform representation (\ref{xi-S}) can then be rewritten as
\begin{align}
\xi(s)=\tfrac{1}{2}\int_{-\infty}^\infty S(\omega)e^{i\omega\tau}\,d\omega=\tfrac{1}{2}\int_{-\infty}^\infty S(\omega)e^{\omega(s-\frac{1}{2})}\,d\omega.
\end{align}
Since the kernel $S(\omega)$ is a symmetric real function, we can readily derive the reflective property $\xi(1-s)=\xi(s)$ and thus $\xi(s)$ is real on the critical line.

The kernel $S_n(\omega)$ of (\ref{S-n}) can be written as
\begin{align}
S_n(\omega)&=8\pi n^2e^{\frac{5\omega}{2}}D_1(n^2e^{2\omega})
\end{align}
with
\begin{align}\label{D-func}
D_1(x)= (\pi x-\tfrac{3}{2})e^{-\pi x}.
\end{align}
Furthermore, we can write $S_n(\omega)$ in terms of $S_1(\omega)$ as follows:
\begin{align}\label{Sn-S1}
S_n(\omega)=\frac{1}{\sqrt{n}}S_1(\omega+\log n), ~~n=1, 2, 3, \ldots.
\end{align}

By substituting (\ref{S-series}) and (\ref{Sn-S1}) into the above, we obtain
\begin{align}\label{new-result}
\xi(s)&=\tfrac{1}{2}\int_{-\infty}^\infty\sum_{n=1}^\infty S_n(\omega)e^{i\omega\tau}\,d\omega
=\tfrac{1}{2}\int_{-\infty}^\infty\sum_{n=1}^\infty\frac{1}{\sqrt{n}}S_1(\omega+\log n)e^{i\omega\tau}\,d\omega\nonumber\\
&=\tfrac{1}{2}\int_{-\infty}^\infty S_1(\omega')e^{i\omega'\tau}\,d\omega'\sum_{n=1}^\infty \frac{1}{\sqrt{n}}e^{-i\tau\log n},
\end{align}
where we set $\omega+\log n=\omega'$ in the above derivation.
The summed term is nothing but the zeta function $\zeta(s)$, i.e.,
\begin{align}
\sum_{n=1}^\infty \frac{1}{\sqrt{n}}e^{-i\tau\log n}=\sum_{n=1}^\infty\frac{1}{n^{\frac{1}{2}+i\tau}}=\zeta\left(\tfrac{1}{2}+i\tau\right)=\zeta(s),
\end{align}

The result (\ref{new-result}) can be compactly expressed as
\begin{align}\label{xi-xi_1-zeta}
\xi(s)&=\xi_1(s)\zeta(s).
\end{align}

By writing
\begin{align}\label{def-g}
g(s)=\frac{s(s-1)}{2}\pi^{-\frac{s}{2}}\Gamma\left(\frac{s}{2}\right),
\end{align}
we can state the following proposition by referring to (\ref{xi-def}):
%
%
\begin{theorem}(The Fourier transform of $\mathbf{S_1(\omega)}$)\\
The function $g(s)$ (\ref{def-g}) that transforms $\zeta(s)$ into $\xi(s)$ by multiplication is the Fourier transform of $S_1(\omega)$ to the domain $\tau$, i.e.,
\begin{align}\label{g-and-S_1}
g(s)=\tfrac{1}{2}\int_{-\infty}^\infty S_1(\omega)e^{i\omega\tau}\,d\omega=\xi_1(s),
\end{align}
where $\tau=t-i\lambda=t-i(\sigma-\tfrac{1}{2})=-i(s-\tfrac{1}{2})$.
\end{theorem}
\begin{proof}
See \cite{kobayashi-report4}.
\end{proof}

Let us denote the Fourier transform of $S_n(\omega)$ as $\xi_n(s)$:
\begin{align}\label{xi_n-def}
\xi_n(s)=\tfrac{1}{2}\int_{-\infty}^\infty S_n(\omega)e^{i\omega\tau}\,d\omega=\xi_1(s)n^{-s},
\end{align}
and
\begin{align}
\xi(s)=\sum_{n=1}^\infty \xi_n(s).
\end{align}
Note that the functions $\xi_n(s)$ are individually complex functions even on the critical line, since $S_n(\omega)$ are not symmetric functions, thus $\xi_n(s)$'s do not enjoy the reflective property that their sum $\xi(s)$ does.  If we define
\begin{align}\label{overline-xi-n}
\overline{\xi}_n(s)=\tfrac{1}{2}[\xi_n(s)+\xi_n(1-s)]=\tfrac{1}{2}[g_n(s)n^{-s}+g_n(1-s)n^{s-1}],
\end{align}
this function is reflective and
\begin{align}
\xi(s)=\sum_{n=1}^\infty \overline{\xi}_n(s).
\end{align}

\subsection{Properties of the $\mathbf{g(s)}$ function}
In this section we discuss some properties of $g(s)$ defined by (\ref{def-g}), and its relations to
the Riemann-Siegel function and Hardy's Z-function.

We set $s=\tfrac{1}{2}+it$ in $g(s)$ and define real functions $a(t)$ and $b(t)$:
\begin{align}
a(t)&=\Re\left\{\log\Gamma\left(\tfrac{1}{4}+\tfrac{it}{2}\right)\right\},\nonumber\\
b(t)&=\Im\left\{\log\Gamma\left(\tfrac{1}{4}+\tfrac{it}{2}\right)\right\}.
\end{align}
Then, we can write
\begin{align}\label{g(s)-critical-line}
g\left(\tfrac{1}{2}+it\right)&=-\tfrac{1}{2}\left(t^2+\tfrac{1}{4}\right)\pi^{-1/4}e^{-i\frac{t}{2}\log\pi}e^{a(t)+ib(t)}.
\end{align}
By defining two real functions $r(t)$ and $\vartheta(t)$
\begin{align}
r(t)&=-\tfrac{1}{2}\left(t^2+\tfrac{1}{4}\right)\pi^{-1/4} e^{a(t)},\nonumber\\
\vartheta(t)&=b(t)-\tfrac{t}{2}\log\pi=\Im\left\{\log\Gamma\left(\tfrac{1}{4}+\tfrac{it}{2}\right)\right\}-\tfrac{t}{2}\log\pi, \label{riemman-siegel}
\end{align}
we can rewrite (\ref{g(s)-critical-line}) as
\begin{align}\label{g-critical-line}
g\left(\tfrac{1}{2}+it\right)&=r(t)e^{i\vartheta(t)}.
\end{align}

The function $\vartheta(t)$ of (\ref{riemman-siegel}) is called the Riemann-Siegel theta function, and the function $Z(t)$ defined by
\begin{align}\label{Z-def}
Z(t)=\zeta\left(\tfrac{1}{2}+it\right)e^{i\vartheta(t)},
\end{align}
is often referred to as Hardy's $Z$-function \cite{ivic}, which is real for real $t$ and has the same zeros as $\zeta(s)$ at $s=\tfrac{1}{2}+it$, with $t$ real.  Thus, locating the Riemann zeros on the critical line reduces to locating zeros on the real line of $Z(t)$.  Furthermore,
\[ |Z(t)|=|\zeta(\tfrac{1}{2}+it)|. \]

Consider the following Stirling approximation formula for $\Gamma(s)$:
\begin{align}
\log \Gamma(s)\approx\tfrac{1}{2}\log\frac{2\pi}{s}+s(\log s -1).
\end{align}
Then
\begin{align}
\log\Gamma(s/2)\approx \left(1-\tfrac{s}{2}\right)\log 2+\tfrac{1}{2}\log\pi+\left(\tfrac{s-1}{2}\right)\log s-\tfrac{s}{2}.
\end{align}
By evaluating the above at $s=\tfrac{1}{2}+it$, we have
\begin{align}
\log\Gamma\left(\tfrac{1}{4}+\tfrac{it}{2}\right)&=a(t)+ib(t)\nonumber\\
&\approx \tfrac{3}{4}\log 2+\tfrac{1}{2}\log\pi-\left(\tfrac{1}{4}+\tfrac{t\theta(t)}{2}\right)-\tfrac{1}{8}\log\left(t^2+\tfrac{1}{4}\right)
 +i\left[\tfrac{t}{4}\log\left(t^2+\tfrac{1}{4}\right)-\tfrac{t}{2}-\tfrac{t}{2}\log 2-\tfrac{\theta(t)}{4}\right],
\end{align}
where
\begin{align}
\theta(t)=tan^{-1}2t.
\end{align}
Thus, we obtain
\begin{align}
r(t)&\approx-2^{-\frac{1}{4}}\pi^{\frac{1}{4}}\left(t^2+\tfrac{1}{4}\right)^{\frac{7}{8}}e^{-\frac{1}{4}-\frac{\theta(t) t}{2}}\nonumber\\
\vartheta(t)&\approx\tfrac{t}{2}\log\tfrac{t}{2\pi e}-\tfrac{\theta(t)}{4}+\tfrac{t}{4}\log\left(1+\tfrac{1}{4t^2}\right).
\end{align}

If we set
\begin{align}\label{A-varphi-def}
A(t)=-r(t),~~\mbox{and}~~\varphi(t)=\vartheta(t)+\pi,
\end{align}
then,
\begin{align}
g\left(\tfrac{1}{2}+it\right)=A(t)e^{i\varphi(t)}.
\end{align}
We denote the real and imaginary parts of $g\left(\tfrac{1}{2}+it\right)$ by $G(t)$ and $\hat{G}(t)$, respectively, viz:
\begin{align}
g\left(\tfrac{1}{2}+it\right)=G(t)+i\hat{G}(t).
\end{align}
Then it is apparent that
\begin{align}
G(t)=A(t)\cos\varphi(t),~~\mbox{and}~~\hat{G}(t)=A(t)\sin\varphi(t).
\end{align}

For sufficiently large $t\gg 1$, $\theta(t)\approx \frac{\pi}{2}$. Thus, $A(t)$ and $\varphi(t)$ can be approximated by
\begin{align}
A(t)&\approx (2e\pi)^{-\frac{1}{4}}e^{-\frac{\pi t}{4}}t^{\frac{7}{4}},~~\mbox{for}~~t\gg 1,\nonumber\\
\varphi(t)&\approx \frac{t}{2}\log\frac{t}{2e\pi}+\frac{7\pi}{8},~~\mbox{for}~~t\gg 1.
\end{align}
The function $A(t)$ is strictly positive for all $t$, hence $G(t)$ becomes zero only when $\varphi(t)=n\pi+\frac{\pi}{2}$ for some integer $n$.  Similarly,
$\hat{G}(t)$ crosses zero only when $\varphi(t)=n\pi$ for integer $n$.
Thus, the number of zeros $N(T)$ of $G(t)$ in $(0, T)$ is given by
\begin{align}\label{N-t}
N(T)=\frac{\varphi(T)}{\pi}\approx \frac{T}{2\pi}\log\frac{T}{2\pi e}+\frac{7}{8},~~T>T(\epsilon).
\end{align}
The same result should hold for the number of zeros $N(T)$ of $\hat{G}(T)$ in $(0,T)$.
The above $N(T)$ agrees to the asymptotic ``Riemann-von Mangoldt formula'' for the number of zeros of $\zeta\left(\tfrac{1}{2}+it\right)$ (and hence the number of zeros of $\xi\left(\tfrac{1}{2}+it\right)$, as well), which Riemann conjectured in his 1859 lecture and proved by von Mangoldt in 1905 (see e.g.,\cite{edwards,matsumoto}).

Gram \cite{gram} observed in 1909 that zeros of $Z(t)$ and zeros of $\sin\vartheta(t)$ alternate on the $t$ axis, with some few exception (see Edwards \cite{edwards} p. 125).  His observation is consistent with our analysis given above that the number of zeros $\hat{G}(t)=A(t)\sin\varphi(t)=-A(t)\sin\vartheta(t)$ in
the interval $[0, t]$ is asymptotically equivalent to that of $\zeta(\tfrac{1}{2}+it)$ (and hence that of $\Xi(t)$ as well).
If we define the complex function
\begin{align}
z(s)=\frac{\xi(s)}{r(t)},
\end{align}
then $z(s)$ is reflective.  Furthermore $z\left(\tfrac{1}{2}+it\right)=Z(t)$, because (\ref{g-critical-line}) and (\ref{Z-def}) imply
\begin{align}
Z(t)=\frac{\Xi(t)}{r(t)}.
\end{align}

Let $G_n(t)$ denote the value  on the critical line of $\overline{\xi}_n(s)$ defined in (\ref{overline-xi-n}), i.e.,
\begin{align}\label{G_n-t}
G_n(t)&=\overline{\xi}_n\left(\tfrac{1}{2}+it\right)=\left.\tfrac{1}{2}[g(s)n^{-s}+g(1-s)n^{s-1}]\right|_{s=\frac{1}{2}+it}
=\tfrac{1}{2}\left[(G(t)+i\hat{G}(t))n^{-\frac{1}{2}-it}+(G(t)-i\hat{G}(t))n^{-\frac{1}{2}+it}\right]\nonumber\\
&=G(t)n^{-\frac{1}{2}}\cos(t\log n)+\hat{G}(t)n^{-\frac{1}{2}}\sin(t\log n)=A(t)n^{-\frac{1}{2}}\cos(\varphi(t)-t\log n).
\end{align}
Thus, we find
\begin{align}
\Xi(t)=\sum_{n=1}^\infty G_n(t)=A(t)\sum_{n=1}^\infty n^{-\frac{1}{2}}\cos(\varphi(t)-t\log n),
\end{align}
where $A(t)=-r(t)$ and $\varphi(t)=\vartheta(t)+\pi$ are defined in (\ref{A-varphi-def}), and
\begin{align}
g\left(\tfrac{1}{2}+it\right)=G(t)+i\hat{G}(t)=A(t)e^{i\varphi(t)}=-r(t)e^{i\vartheta(t)}.
\end{align}

\appendix
\numberwithin{equation}{section}
\renewcommand\thefigure{\thesection.\arabic{figure}}
\counterwithin{figure}{section}

\section{Derivation of (\ref{xi-def}) and (\ref{edwards-p16})}\label{append-A}
Although the essence of both equations is found in Riemann's original paper, we follow Edwards \cite{edwards} and
Matsumoto \cite{matsumoto}. We begin with the integral representation of the gamma function
\begin{align}
\Gamma(s)=\int_0^\infty u^{s-1}e^{-u}\,du.
\end{align}
By setting $u=\pi n^2x$, we have
\begin{align}
\Gamma(s)=\pi^sn^{2s}\int_0^\infty x^{s-1}e^{-\pi n^2 x}\,dx.
\end{align}
Then,
\begin{align}
\Gamma(s/2)=\pi^{s/2}n^{s}\int_0^\infty x^{\frac{s}{2}-1}e^{-\pi n^2 x}\,dx,
\end{align}
from which we obtain
\begin{align}
\pi^{-s/2}\Gamma(s/2)n^{-s}=\int_0^\infty x^{\frac{s}{2}-1}e^{-\pi n^2 x}\,dx.
\end{align}
By summing up over $n$ from 1 to infinity, we obtain
\begin{align}\label{nu}
\pi^{-s/2}\Gamma(s/2)\zeta(s)=\int_0^\infty x^{\frac{s}{2}-1}\psi(x)\,dx,
\end{align}
where $\psi(x)$ is given defined in (\ref{func-psi}).

Let us write (\ref{nu}) as $\nu(s)$, and the split the integration interval of the RHS into the two subintervals, $[0,1)$ and $[1,\infty)$, viz:
\begin{align}
\nu(s)&=\pi^{-s/2}\Gamma(s/2)\zeta(s)=\int_0^\infty x^{\tfrac{s}{2}-1}\psi(x)\,dx\nonumber\\
&=\int_0^1 x^{\tfrac{s}{2}-1}\psi(x)\,dx +\int_1^\infty x^{\frac{s}{2}-1}\psi(x)\,dx.
\end{align}
By substituting Jacobi's identity for $\psi(x)$ given by (\ref{Jacobi}) into the first integrand, we find
\begin{align}\label{nu-function}
\nu(s)&=\int_0^1 x^{\tfrac{s}{2}-1}\left(x^{-1/2}\psi(x^{-1})+\tfrac{1}{2}x^{-1/2}-\tfrac{1}{2}\right)\,dx +\int_1^\infty x^{\frac{s}{2}-1}\psi(x)\,dx\nonumber\\
&=-\frac{1}{1-s}-\frac{1}{s}+\int_1^\infty\left(x^{\frac{s}{2}-1}+x^{\frac{1-s}{2}-1}\right)\psi(x)\,dx.
\end{align}
It is apparent that $\nu(s)$ satisfies the reflective property, i.e.,
\[ \nu(1-s)=\nu(s).\]
The function $\nu(s)$ is not an entire function since it has $s=0$ and $s=1$ as poles.  By multiplying $\nu(s)$ by $\displaystyle{-\frac{s(1-s)}{2}}$, we define $\xi(s)$, viz.
\begin{align}
\xi(s)=-\tfrac{1}{2}s(1-s)\nu(s)=\tfrac{1}{2}s(s-1)\pi^{-s/2}\Gamma(s/2)n^{-s}\zeta(s),
\end{align}
which is (\ref{xi-def}).

The function $\xi(s)$ should satisfy the reflective property (\ref{symmetric_property}) since both $\nu(s)$ and $\displaystyle{-\frac{s(1-s)}{2}}$ are reflective.
From (\ref{nu-function}), we obtain
\begin{align}
\xi(s)=\tfrac{1}{2}-\tfrac{1}{2}s(1-s)\int_1^\infty\left(x^{\frac{s}{2}}+x^{\frac{1-s}{2}}\right)\psi(x)\,\frac{dx}{x},
\end{align}
which is (\ref{edwards-p16}).  From the last expression, it is apparent that $\xi(0)=\xi(1)=\tfrac{1}{2}$.
\end{document}